\documentclass[11p,reqno]{amsart}

\usepackage[T1]{fontenc}
\usepackage{graphicx}
\usepackage{amssymb,amsthm,amsmath,mathrsfs,bm,braket,marginnote}
\usepackage{enumerate}
\usepackage{appendix}
\usepackage[colorlinks=true, pdfstartview=FitV, linkcolor=blue, citecolor=blue, urlcolor=blue]{hyperref}
\usepackage{multirow}

\usepackage{pgf}
\usepackage{pgfplots}
\usepackage{tikz}
\usetikzlibrary{arrows,calc}
\usepackage{verbatim}
\usetikzlibrary{decorations.pathreplacing,decorations.pathmorphing}
\usepackage[numbers,sort&compress]{natbib}
\usepackage{dsfont}
\numberwithin{equation}{section}
\newtheorem{theorem}{Theorem}[section]

\newtheorem{proposition}[theorem]{Proposition}

\theoremstyle{remark}

 \reversemarginpar

\begin{document}

\title[]{Local central limit theorem for real eigenvalue fluctuations of elliptic {G{\SMALL in}OE} matrices}

\author{Peter J. Forrester}
\address{School of Mathematical and Statistics, The University of Melbourne, Victoria 3010, Australia}
\email{pjforr@unimelb.edu.au}

\date{}

\dedicatory{}


\begin{abstract}
Random matrices from the elliptic Ginibre orthogonal ensemble (GinOE) are a certain linear combination of
a real symmetric, and real anti-symmetric, real Gaussian random matrices and controlled by a parameter
$\tau$. Our interest is in the fluctuations of the number of real eigenvalues, for fixed $\tau$ when
the expected number is proportional to the square root of the matrix size $N$, and for $\tau$ scaled
to the weakly non-symmetric limit, when the number of eigenvalues is proportional to $N$. By establishing
that the generating function for the probabilities specifying the distribution of the number of real
eigenvalues has only negative real zeros, and using too the fact that variances in both circumstances
of interest tends to infinity as $N \to \infty$, the known central limit theorem for the fluctuations is strengthened
to a local central limit theorem, and the rate of convergence is discussed.
\end{abstract}

\maketitle

\section{Introduction}

A property of non-Hermitian random matrices with real entries is that in general real eigenvalues occur with non-zero
probability. The first such random matrices to be studied in detail from this viewpoint was the case of standard Gaussian
entries \cite{EKS94,Ed97}. This set of random matrices is now referred to as the Ginibre orthogonal ensembles (GinOE); see the recent review
\cite{BF23}. Let $N_{\mathbb R}$ denote the random variable for the number of real eigenvalues. An early finding was for
the large matrix size $N$  of the expected value of $N_{\mathbb R}$  \cite{EKS94},
\begin{equation}\label{E0}
 \mathbb E (N_{\mathbb R})    \mathop{\sim}\limits_{N \to \infty} \sqrt{ {2 N \over \pi}}.
   \end{equation} 
   Later the large $N$ form of the corresponding variance was shown to be proportional to expected value  \cite{FN07},
  \begin{equation}\label{E1} 
  \sigma^2(N_{\mathbb R})   \mathop{\sim}\limits_{N \to \infty} (2 - \sqrt{2})  \mathbb E (N_{\mathbb R}).
   \end{equation}   

Ask now about the large $N$ form of the scaled distribution
\begin{equation}\label{1.1}
{\rm Pr} \Big ( { N_{\mathbb R}  - \mathbb E (N_{\mathbb R})  \over \sigma(N_{\mathbb R})}  \le x  \Big ).
\end{equation}
As part of a more general study probing the asymptotic distribution of a scaled polynomial
linear statistic $\sum_{j=1}^N p(\lambda_j/\sqrt{N})$  
\cite{Si17}, and extended to a more general class
of test functions in \cite{FS21}, it was proved (choose $p(x) = 1$) that (\ref{1.1}) limits to a standard normal distribution. Equivalently
this establishes the central limit theorem (CLT) for the scaled fluctuation of $N_{\mathbb R}$,
\begin{equation}\label{1.1a}
\lim_{N \to \infty}  { N_{\mathbb R}  - \mathbb E (N_{\mathbb R})  \over \sigma(N_{\mathbb R}) } \mathop{=}\limits^{\rm d} {\rm N}[0,1].
 \end{equation}

 Recently, as part of the review \cite{BF23}, the CLT (\ref{1.1a}) was strengthened to a local central
 limit theorem (LCLT). To state the latter, let $p_{2k,N}$ denote the probability that an $N \times N$ GinOE matrix
 has, for $N$ even, exactly $2k$ real eigenvalues (since complex eigenvalues of real matrices occur in complex
 conjugate pairs, the parity of the matrix size and the number of real eigenvalues must agree). 
 Note that $\{p_{2k,N}\}$ define the distribution of $ N_{\mathbb R}$. With this notation
the LCLT in question reads  \cite[Prop.~2.4]{BF23} 
 \begin{equation}\label{2.1}
 \lim_{N \to \infty} {\rm sup}_{x \in (-\infty,\infty)} \Big |  \sigma(N_{\mathbb R})  p_{2k,N}\ |_{2k = \lceil \sigma(
 N_{\mathbb R}) x + \mathbb E (N_{\mathbb R}) \rceil} - {1 \over \sqrt{2 \pi}} e^{-x^2/2}\Big | = 0.
  \end{equation} 
 This a stronger convergence statement than the CLT (\ref{1.1a}), giving precise information about the individual
 probabilities, whereas the former gives an asymptotic equality for the cumulative sum of the probabilities.
 
 The purpose of this note is to generalise (\ref{2.1}) to the setting of the real eigenvalues for the elliptic
 GinOE class of random matrices. To define the latter, for $G_1,G_2 \in {\rm GinOE}$ we define a real symmetric matrix $S$, and real
 anti-symmetric matrix $A$, by setting $S= {1 \over 2}(G_1 + G_1^T)$ and $A= {1 \over 2} (G_2 - G_2^T)$.
 With $\tau$ a parameter, $0 < \tau < 1$, a member $X$ of the elliptic GinOE is defined in terms of $S$ and $A$ according to
 the linear combination
  \begin{equation}\label{SA}
 X = \sqrt{1 + \tau} S + \sqrt{1 - \tau} A.
 \end{equation}
 Note that with $\tau = 0$ a GinOE matrix is obtained, while for $\tau = 1$ the matrix is  real symmetric and in fact
 a member of the Gaussian orthogonal ensemble (see e.g.~\cite[\S 1.1]{Fo10}).
 With $X \mapsto X/\sqrt{N}$, it is known \cite{DGIL94,FJ96} that for $N \to \infty$ the eigenvalue density is
 supported in an ellipse with semi-axes $A = 1 + \tau$, $B= 1 - \tau$, and moreover upon normalising to integrate to unity
 has the constant value
 $1/(\pi(1 - \tau^2))$ in this support. It is also known that setting $\tau = 1 - \alpha^2/N$, with $\alpha > 0$ fixed, allows for 
 a well defined weakly non-symmetric limit  \cite{FKS97}.
 
 For the elliptic GinOE (indicated in the notation below by the superscript $\tau$), one has from \cite{FN08p}
  \begin{equation}\label{3.0a}
 \mathbb E (N_{\mathbb R}^\tau)    \mathop{\sim}\limits_{N \to \infty} \sqrt{ {2 \over \pi} {1 + \tau \over 1 - \tau} N}.
   \end{equation} 
 In the weakly non-symmetric limit this leading $\sqrt{N}$ behaviour  is replaced by an asymptotic form
 proportional to $N$ \cite{BKLL23}
 \begin{equation}\label{3.0b} 
  \mathbb E (N_{\mathbb R}^\tau) \Big |_{\tau = 1 - \alpha^2/N}   \mathop{\sim}\limits_{N \to \infty}  c(\alpha) N, \qquad c(\alpha) := e^{-\alpha^2/2} \Big ( I_0(\alpha^2/2) +
 I_1(\alpha^2/2)   \Big ).
   \end{equation} 
   Here $I_\nu(z)$ denotes the purely imaginary argument Bessel function. Also, for the corresponding variances we have \cite{BKLL23}
  \begin{equation}\label{3.0c} 
  \sigma^2(N_{\mathbb R}^\tau)    \mathop{\sim}\limits_{N \to \infty}  (2 - \sqrt{2})  \mathbb E (N_{\mathbb R}^\tau)
   \end{equation}   
   and
  \begin{equation}\label{4.0} 
  \sigma^2(N_{\mathbb R}^\tau)   \Big |_{\tau = 1 - \alpha^2/N}   \mathop{\sim}\limits_{N \to \infty}    \Big ( 2 - 2 {c(\sqrt{2} \alpha) \over c(\alpha)} \Big )  \mathbb E (N_{\mathbb R}^\tau).
  \end{equation}   
  A very recent result of Byun, Molag and Simm \cite{BMS23} gives the CLTs
   \begin{equation}\label{4.1} 
 \lim_{N \to \infty} { N_{\mathbb R}  - \mathbb E (N_{\mathbb R})  \over \sigma(N_{\mathbb R}) }\mathop{=}\limits^{\rm d} {\rm N}[0,1], \qquad 
  \lim_{N \to \infty}   { N_{\mathbb R}  - \mathbb E (N_{\mathbb R})  \over \sigma(N_{\mathbb R}) } \Big |_{\tau = 1 - \alpha^2/N}  \mathop{=}\limits^{\rm d} {\rm N}[0,1] .
  \end{equation}   
  Our main result extends (\ref{4.1}) to a LCLT.
  
  \begin{theorem}\label{T1}
  Let $N$ be even and let $ p_{2k,N}^\tau$ denote the probability that there are exactly $2k$ real eigenvalues for a matrix drawn from the elliptic
  GinOE. We have that $\{ p_{2k,N}^\tau \}$ satisfy the LCLT
   \begin{equation}\label{5.1}
\lim_{N \to \infty} {\rm sup}_{x \in (-\infty,\infty)} \Big |  \sigma(N_{\mathbb R}^\tau)  p_{2k,N} \, |_{2k = \lceil \sigma(
 N_{\mathbb R}^\tau) x + \mathbb E (N_{\mathbb R}^\tau) \rceil} - {1 \over \sqrt{2 \pi}} e^{-x^2/2} \Big | = 0.
  \end{equation}      
With $ \mathbb E (N_{\mathbb R}^\tau)$, $  \sigma(N_{\mathbb R}^\tau) $ appropriately specified, this holds both for $\tau$ fixed
with respect to $N$, and for the weakly non-symmetric scaling $\tau = 1 - \alpha^2/N$.
\end{theorem}

 \section{Earlier examples of LCLTs in random matrix theory and methodology}\label{S3}
 To this author's knowledge, the first example of a LCLT identified in random matrix theory was in a setting
 analogous to that of Theorem \ref{T1} \cite{FM11}. Thus for $G_1,G_2 \in {\rm GinOE}$ form the spherical GinOE matrix $G_1^{-1} G_2$, so named
 since upon a stereographic projection its eigenvalues are naturally associated with a point process on the sphere \cite{EKS94}. For
 $N$ even let $p_{2k,N}^{\rm s}$ denote the probability that there are exactly $k$ real eigenvalues for a matrix drawn from the spherical GinOE,
 and introduce the generating function
 $$
 Z^{\rm s}_N(\zeta) := \sum_{k=0}^{N/2} \zeta^k p_{2k,N}^{\rm s}.
 $$
 Note that with $\zeta = e^{it}$ this has the interpretation as the characteristic function for the distribution of $\{ p_{2k,N}^{\rm s} \}$.
 In \cite{FM11} the factorisation formula 
    \begin{equation}\label{6.1}
  Z^{\rm s}_N(\zeta) = A_N \prod_{l=1}^{N/2} (1 + \zeta/u_l(N)), \qquad u_l(N) > 0,
   \end{equation} 
   for explicit $A_N,   u_l(N)$ was given, establishing in particular that all the zeros of $ Z^{\rm s}_N(\zeta)$  lie
   on the negative real axis in the complex $\zeta$-plane. We remark that this  implies  $\{ p_{2k,N}^{\rm s} \}$ is  an example of a P\'olya frequency sequence \cite{Sc55,Pi97}.
   
   There are at least two significant consequences of this property with regards to the distribution of $\{ p_{2k,N}^{\rm s} \}$. One is that a CLT holds,
   while the other is that the CLT can be strengthened to a LCLT. In relation to the CLT, for $0 \le \lambda_j \le 1$ $(j=1,\dots,N/2)$ denote by ${\rm Ber}[\lambda_j]$ the
   Bernoulli distribution for a random variable $x_j \in \{0,1\}$ specified by ${\rm Pr}(x_j = 1) = \lambda_j$, ${\rm Pr}(x_j = 0) = 1 - \lambda_j$
   $(0 < \lambda_j < 1)$. Consider the sum of $N/2$ ($N$ even) such random variables 
     \begin{equation}\label{6.2}
   \sum_{j=1}^{N/2} x_j \mathop{=}\limits^{\rm d} \sum_{j=1}^{N/2} {\rm Ber}[\lambda_j].
    \end{equation} 
    The characteristic function of this sum is
     \begin{equation}\label{7.1}   
  \prod_{j=1}^{N/2} (1 - \lambda_j + e^{it} \lambda_j).
   \end{equation} 
   Comparing (\ref{7.1}) with (\ref{6.1}) shows that we can identify $\zeta = e^{it}$ and $u_l(N) = (1 - \lambda_l)/\lambda_l$, with the requirement that
   $u_l(N) > 0$ being upheld since $\lambda_l$ takes on values between $0$ and $1$. For the sum (\ref{6.2}), standard arguments
    \cite[Section XVI.5, Theorem 2]{Fe66},  \cite{Ha67},  \cite{Ca80} gives that a CLT holds for $N \to \infty$ whenever the corresponding variance diverges
    in this limit. On this latter point, it can be established that $\sigma(N_{\mathbb R}^{\rm s}) \to \infty$ as $N \to \infty$ \cite{FM11} (in fact structurally the same
    asymptotic relation as (\ref{3.0c}) holds true, with the asymptotic value of the expectation on the RHS $\mathbb E (N_{\mathbb R}^{\rm s})$ now
    being given by $\sqrt{\pi N/2}$ which is $\pi/2$ times (\ref{E0})). Hence, by the Bernoulli random variable interpretation of (\ref{6.1}), we can conclude the validity of the
    CLT (\ref{4.1}) with $N_{\mathbb R}$ replace by $N^{\rm s}_{\mathbb R}$.
    
    We now come to the consequence of the factorisation (\ref{6.1}) in relation to the validity of a LCLT. By a theorem attributed to Newton relating to a
    property of the elementary symmetric polynomials (here in the variables $\{ u_l(N) \}$), see e.g.~\cite{Ni00}, the fact that  $u_l(N) \ge 0$ ($l=1,\dots,
    N/2$) implies that the coefficients of the series expansion of $Z_N^{\rm s}(\zeta)$, i.e.~$\{ p_{2k,N}^{\rm s} \}$, form a log-concave sequence,
  \begin{equation}\label{C1} 
  \log p_{2(k+1),N}^{\rm s} - 2      \log p_{2k,N}^{\rm s} +   \log p_{2(k-1),N}^{\rm s}      \le 0.
  \end{equation} 
  It is proved in \cite[Th.~2]{Be73} that (\ref{C1}) together with the CLT (\ref{7.1}) are sufficient for the validity of  
   the LCLT (\ref{2.1}) with $N_{\mathbb R}$ replaced by $N_{\mathbb R}^{\rm s}$.
   
   Subsequent to \cite{FM11}, in \cite{FL14} a related but distinct setting in random matrix theory giving rise to LCLTs was identified.
   This is in relation to the random variable $N(\Lambda)$ for the number of eigenvalues in a region $\Lambda$, in the circumstance that
   the volume $| \Lambda| \to \infty$ and that the eigenvalues for a determinantal point process with an Hermitian kernel (for more on the latter
   see e.g.~\cite{Bo11}). As an explicit example, consider the Ginibre unitary ensemble (GinUE) of standard complex Gaussian entries.
   In the limit $N \to \infty$ the eigenvalues form a determinantal point process with Hermitian kernel
    \begin{equation}\label{K1}
   {1 \over \pi} e^{-(|w|^2 + |z|^2)/2} e^{w \bar{z}};
    \end{equation} 
   see the recent review \cite{BF22}. Let $E_k(0;\Lambda)$ denote the probability that there are no eigenvalues in the region $\Lambda$,
   with $\{ E_k(0;\Lambda) \}$ specifying the random variable $N(\Lambda)$. The fact that (\ref{K1}) is an Hermitian kernel can be used to show that
   the generating function $\sum_{k=0}^\infty z^k E_k(0;\Lambda)$ only has negative real zeros. With it being known that the corresponding variance
   $\sigma^2(N(\Lambda))$ tends to infinity as $|\lambda| \to \infty$ \cite{MY80} (see too the recent review \cite[\S 2.9]{Fo23}), the same reasoning as used in the paragraph containing (\ref{6.2}) giving the  CLT (\ref{4.1}) with $N^{\rm e}_{\mathbb R}$ replace by $N^{\rm s}_{\mathbb R}$, and in the above paragraph containing (\ref{C1}) strengthening this to a LCLT, again applies. Consequently a LCLT theorem holds for the
   probabilities $ E_k(0;\Lambda)$ (replace $ p_{2k,N}^{\tau} $ with $ E_k(0;\Lambda)$, and $N_{\mathbb R}^{\tau}$ by    $N(\Lambda)$ in (\ref{5.1})).
   
 We see then that the main task in establishing a LCLT according to this strategy, after establishing that the variance diverges as $N \to \infty$,
  is to be able to show that the generating function for the underlying
 probabilities has all its zeros on the negative real axis. For the spherical GinOE this was immediate by knowledge of the explicit factorisation
 (\ref{6.1}).  In the recent review \cite{BF23}, a method independent of an explicit factorisation was introduced
  in the case of the
 $\{ p_{2k,N} \}$ for GinOE. Specifically, this was shown as a consequence of the determinant formula \cite{KPTTZ16}
  \begin{equation}\label{10.1}
  \sum_{k=0}^{N/2} z^k p_{2k,N} = \det \bigg [
  \delta_{j,k} + {(z-1) \over \sqrt{2 \pi}} {\Gamma(j+k-3/2) \over \sqrt{\Gamma(2j-1) \Gamma(2k-1)}} \bigg ]_{j,k=1}^{N/2}.
   \end{equation} 
   Thus let $\{ z_p \}_{p=1,\dots,N/2}$ denote the zeros of (\ref{10.1}). We observe that $\{1/(1-z_p) \}_{p=1,\dots,N/2}$ are the eigenvalues of the matrix
  \begin{equation}\label{11.0}
  \bigg [ {1 \over \sqrt{2 \pi}} {\Gamma(j+k-3/2) \over \sqrt{\Gamma(2j-1) \Gamma(2k-1)}} \bigg ]_{j,k=1}^{N/2}.
   \end{equation} 
   This is a real symmetric matrix and so all the eigenvalues are real, and hence so are the zeros $\{z_p\}$. Moreover, since $\{ p_{2k,N} \}$ are
   probabilities, these zeros are both real and non-positive (in fact since as established in \cite{KPTTZ16} $ p_{0,N} > 0$, which tells us that all
   the zeros are in fact negative), which is the required result.
   
   Let's return now to the consideration of a LCLT for $\{ p_{2k,N}^{\rm \tau}\}$ as relates to the random variable $N_{\mathbb R}^\tau$.
   The asymptotic formulas (\ref{3.0c}) and (\ref{4.0}) tell us that $\sigma^2(N_{\mathbb R}^\tau) \to \infty$ as $N \to \infty$ in both cases of interest,
    $\tau$ fixed and $\tau = 1 - \alpha^2/N$. Thus to give a proof
    of Theorem \ref{T1} it is sufficient to establish that the generating function
      \begin{equation}\label{12.1}
  Z_N^{\tau}(z) := \sum_{k=0}^{N/2} z^k p_{2k,N}^{\tau}
    \end{equation} 
    has all its zeros real and non-positive in the complex $z$-plane. We have seen in (\ref{10.1}) for the GinOE probabilities $\{ p_{2k,N} \}$
    that an avenue to deduce the location of the zeros is through a determinant formula involving a real symmetric matrix. In fact, as to be
    revised and discussed in the next section, such a formula has been given in the recent work \cite{BMS23}.

 \section{The zeros of $Z_N^{\tau}(z)$ and a proof of Theorem \ref{T1}}
 Let $H_n(x)$ denote the Hermite polynomials. The following determinant formula for $Z_N^{\tau}(z)$ holds.
 
 \begin{proposition}\label{Px} (Prop.~2.1 of \cite{BMS23})
 We have
  \begin{equation}\label{12.1a}
  Z_N^{\tau}(z) =  \det \Big [
  \delta_{j,k} +  (z-1) M_N^{\tau}(j,k) \Big ]_{j,k=1,\dots,N/2},
 \end{equation} 
 where
    \begin{equation}\label{13.0}
 M_N^{\tau}(j,k)      =  {1 \over \sqrt{2 \pi}} {(\tau/2)^{j+k-2} \over \sqrt{\Gamma(2j-1) \Gamma(2k - 1)}}
 \int_{-\infty}^\infty e^{-x^2/(1 + \tau)} H_{2j-2}(x/\sqrt{2 \tau})  H_{2k-2}(x/\sqrt{2 \tau}) \, dx.
 \end{equation} 
 \end{proposition}
 
 \begin{proof}
 It is remarked in \cite{BMS23} that (\ref{12.1a}) is equivalent to a determinant formula for
 $p_{2k,N}^{\tau}$ given in \cite[Eq.~(3.8)]{FN08p}. Not starting from the latter formula directly,
 but using other formulas from \cite{FN08p}, working is given in \cite{BMS23} to
 deduce (\ref{12.1a}). Due to the importance of this result to the proof of
 Theorem \ref{T1}, we will give a derivation here too, which in comparison to the one in
 \cite{BMS23} differs for the emphasis we place on a derivative structure present in the 
 underlying skew orthogonal polynomials.
 
 We will take as our starting point a minor variation of \cite[Eq.~(3.8) with the change of variables $x \mapsto x/\sqrt{1+\tau}$,
 $y \mapsto y/\sqrt{1 + \tau}$]{FN08p}, which states
    \begin{equation}\label{13.1}
  Z_N^{\tau}(z) = {(1 + \tau)^{-N/2}     \over 2^{N(N+1)/2} \prod_{l=1}^N \Gamma(l/2)}
  \det [ z \alpha_{2j-1,2k} + \beta_{2j-1,2k} ]_{j,k=1,\dots,N/2}.
   \end{equation} 
   Here, with $\{p_n(x) \}$ a set of monic polynomials of the indexed degree, and further having the
   same parity under the mapping $x \mapsto - x$ as the index, the quantities in the determinant are specified by
     \begin{align}\label{pm}
    & \alpha_{j,k} = \int_{-\infty}^\infty dx \int_{-\infty}^\infty dy \, e^{-(x^2 + y^2)/(2(1 + \tau))} p_{j-1}(x)  p_{k-1}(x) 
     {\rm sgn}\, (y-x) \nonumber \\
     & \beta_{j,k} =  2i \int_0^\infty dx \int_{-\infty}^\infty dy \, e^{(y^2 - x^2)/(1 + \tau)} {\rm erfc} \Big ( \sqrt{2 \over 1 - \tau^2} y \Big ) \nonumber \\
     & \qquad \qquad \times \Big ( p_{j-1}(x+iy)  p_{k-1}(x-iy) - ( p_{k-1}(x+iy)  p_{j-1}(x-iy) \Big ).
 \end{align}
 (In \cite[Eq.~(3.8)]{FN08p} the particular choice of monic polynomials with the required parity property $p_j(x) = x^j$ is made, but this
 is not what we want here.) 
 
 From the definitions one sees that
 $$
 \langle p_{j-1}, p_{k-1} \rangle := \alpha_{j,k} + \beta_{j,k}
 $$
 defines a skew-symmetric inner product. We know from \cite[Th.~1 and Eq.~(4.40)]{FN08p} that with
 $$
 C_k(z) := \Big ( {\tau \over 2} \Big )^{k/2} H_k(z/\sqrt{2 \tau}),
 $$
 the choice of monic polynomials $p_j(z) = q_j(z)$, where
   \begin{equation}\label{q}
q_{2k}(z) = C_{2k}(z), \qquad q_{2k+1}(z) = - (1 + \tau) e^{z^2/(2(1+\tau))} {d \over dz} \Big (
 e^{-z^2/(2(1+\tau))}   C_{2k}(z) \Big ),
   \end{equation} 
   have the  skew orthogonality property
   \begin{equation}\label{S.1}
   \langle q_{2j}, q_{2k} \rangle   =  \langle q_{2j+1}, q_{2k+1} \rangle = 0 \: \: (j,k=0,1,\dots), \qquad
     \langle q_{2j}, q_{2k+1} \rangle = 0 \: \: (j,k=0,1,\dots (j \ne k)) 
   \end{equation} 
   with the normalisation
   \begin{equation}\label{S.2} 
    \langle q_{2j}, q_{2j+1} \rangle      = (2j)! 2 \sqrt{2 \pi} (1 + \tau).
  \end{equation} 
  
  The monic skew orthogonal polynomials (\ref{S.1}) have the required parity property and so can be substituted in
  (\ref{pm}). Doing this, and substituting the result in (\ref{13.1}), simple manipulation using the skew orthogonality
  property shows
   \begin{equation}\label{S.3}
  Z_N^{\tau}(z) = 
  \det \bigg [ \delta_{j,k} + (z-1) {1 \over 2 \sqrt{2\pi}} {(1+\tau) \over \sqrt{\Gamma(2j-1) \Gamma(2k-1)} } \alpha_{2j-1,2k} \bigg ]_{j,k=1,\dots,N/2}.
   \end{equation}   
   From the definition in (\ref{pm}) of $ \alpha_{2j-1,2k}$, together with the derivative formula for $q_{2k+1}(z)$ as given in 
   (\ref{q}), the double integral in the former can be reduced to a single formula using integration by parts. This reduces (\ref{S.3})
   to the form (\ref{12.1a}), as required.

 \end{proof}   
 
 We are now in a position to establish the required property of the zeros of $Z_N^\tau(z)$ and thus to
 conclude the proof of Theorem \ref{T1}.
 The matrix 
   \begin{equation}\label{16.0}
   M_N^{\tau} := [  M_N^{\tau}(j,k)  ]_{j,k=1,\dots,N/2}
  \end{equation} 
  in (\ref{12.1a}) is real symmetric and so its eigenvalues are real (in fact they are positive
  since it is easy to verify from (\ref{13.0}) that   $M_N^{\tau} $ is positive definite). The
  reasoning of the paragraph containing (\ref{11.0}) tells us that all the zeros of $Z_N^\tau(z)$ are on
  the negative real axis in the complex $z$-plane. As noted in the final paragraph of Section \ref{S3}, this
  implies the validity of Theorem \ref{T1}.
  
  \section{Numerical calculations and discussion}
  The integral expression (\ref{13.0}) defining the matrix elements $M_N^\tau(j,k)$ is shown in \cite[Prop.~2.1]{BMS23} to permit an
  evaluation in terms of the ${}_2 F_1$ hypergeometric function
  \begin{multline}\label{16.1}  
 M_N^{\tau}(j,k) =   {1 \over 2 \sqrt{2\pi}}  \Big ( {1 + \tau \over 1 - \tau} \Big )^{1/2}
 {\Gamma(k-j-3/2) \over  \sqrt{\Gamma(2j-1) \Gamma(2k-1)} } \\
 \times {}_2 F_1(k-j+1/2,j-k+1/2;-j-k+5/2;-\tau/(1- \tau) ).
 \end{multline}
 With $\{ \lambda_j^\tau(N) \}_{j=1,\dots,N/2}$ the eigenvalues of the matrix $M_N^\tau$ (\ref{16.0}), it follows from 
 (\ref{12.1}) that
   \begin{equation}\label{16.1a}
  Z_N^{\tau}(z) = \prod_{l=1}^{N/2} ( 1 + (z - 1) \lambda _j^\tau(N) ).
  \end{equation} 
  
  For given $\tau$, the matrix elements in (\ref{16.1}) can be computed to high precision, using the computer algebra software
  Mathematica for example. With $N$ fixed, this allows the  eigenvalues of $ M_N^{\tau}$ to be computed to high precision. Hence by
  (\ref{16.1a}) $Z_N^{\tau}(z) $ is known in product form. The coefficients in the series expansion of this product form are,
  by (\ref{12.1}),   $\{ p_{2k,N}^{\rm \tau}\}$. These can be extracted from the formula for a Fourier coefficient of a Fourier series
  $$
  p_{2k,2M}^{\rm \tau} = {1 \over M} \sum_{l=0}^{M-1}  Z_N^{\tau}( e^{2 \pi i l /M})  e^{-2 \pi i l  k/M},
  $$
  where we have set $M = N/2$ for convenience. This gives rise to high precision evaluation of $\{ p_{2k,N}^{\rm \tau}\}$, as can be
  checked from the requirement that the probabilities sum to 1. As an example
   \begin{equation}\label{E.1}
 p_{2k,N}^{\rm \tau} \Big |_{N = 80, k = 20 \atop \tau = 1/2} = 7.946014632966 \cdots \times 10^{-23}.
   \end{equation} 
   
   It turns out that knowledge of the eigenvalues can also be used to  compute a large $N$ Stirling-type asymptotic expression for
 $\{ p_{2k,N}^{\rm \tau}\}$ \cite{Ha56,Pi97}. For $r>0$ define
  \begin{equation}\label{A.r}
  a(r) = r {d \over dr}  Z_N^{\tau}(r) = r \sum_{l=1}^{N/2} {\lambda_l^\tau(N) \over 1 + (r-1) \lambda_l^\tau(N)}, \quad
  b(r) = r {d \over dr}  a(r) = r \sum_{l=1}^{N/2} {\lambda_l^\tau(N)(1 - \lambda_l^\tau(N)) \over (1 + (r-1) \lambda_l^\tau(N) )^2}.
   \end{equation}  
   These correspond to the mean and variance of the discrete probability sequence $\{ r^k p_{2k,N}^{\rm \tau}/ Z_N^\tau(r) \}_{k=0}^{N/2}$.
   In the setting that all the zeros of $Z_N^\tau(z)$ are on the negative real axis in the complex $z$-plane, one has that the equation
   $a(r) = k$ has a unique positive solution $r_k$ say; see \cite[below (18)]{Pi97}. The Stirling-like formula is in terms of $b(r)$ in (\ref{A.r}),
   the generating function $Z_N^\tau(r)$ and $r_k$, and reads \cite[Eq.~(26)]{Pi97}
     \begin{equation}\label{P.r}
  p_{2k,N}^{\rm \tau}  = {1 \over \sqrt{2 \pi b(r_k)}} {Z_N^\tau(r_k) \over (r_k)^k} (1 + \varepsilon_k).
   \end{equation}  
   Here $|\varepsilon_k|$ is bounded by $C/b(r_k)^{1/2}$ for some $C>0$ independent of $k$ (in \cite[text below (27)]{Pi97} it is suggested
   that in the bound $b(r_k)^{1/2}$ can be replaced by $b(r_k)$). Computing the approximation (\ref{P.r}) for the setting of (\ref{E.1}) gives
    \begin{equation}\label{E.2}
 p_{2k,N}^{\rm \tau} \Big |_{N = 80, k = 20 \atop \tau = 1/2} \approx 7.943 \times 10^{-23}.
  \end{equation}  
  The ratio of the approximation (\ref{E.2}) to the exact value (\ref{E.1}) is $0.9996$. On the other hand $b(r_k)$ evaluates to $1.249\cdots$ so
  the bound(s) noted below (\ref{P.r}) is not informative apart from the sign. We remark that a Stirling-like formula analogous to (\ref{E.2}) has been used in the recent work
  \cite{Bo23} in the context of a study of corrections to the known random matrix limit for the distribution of the longest increasing subsequence length of
  a random permutation. We remark that for $\tau = 0$ (GinOE case) and $k$ proportional to $N$, $k/N = \mu$ say, the asymptotic formula
  $p_{2k,N}^{\rm \tau} |_{\tau = 0} \sim e^{-c(\mu) N^2}$ has been established, where the proportionality $c(\mu)$ can be interpreted as an 
  energy for a certain two-dimensional electrostatics problem \cite{GPTW16}.

  Of interest is the rate of convergence to the limit law implied by (\ref{5.1}), and thus a bound on the difference
     \begin{equation}\label{E.3}
     p_{2k,N} - {1 \over \sqrt{2 \pi \sigma^2(N_{\mathbb R}^\tau)}} e^{ -  (2k -  \mathbb E(N_{\mathbb R}^\tau ))^2 /(2 \sigma^2(N_\mathbb R))}.
 \end{equation}
 In fact \cite[Eq.~(25)]{Pi97} gives that there exists a $C>0$ such that the absolute value of (\ref{E.3}) is less than $C/\sigma^2(N_{\mathbb R}^\tau)$ for each
 $k=0,1,\dots,N/2$. This bound is evident from numerical computation (see Figure \ref{Fig1-23}), which moreover suggests that (\ref{E.3}) when  multiplied by
 $\sigma^2(N_{\mathbb R}^\tau)$ tends to a well defined limiting functional form. See \cite{FM23,Bo23,Bo23a} for the recent identification and study
 of an analogous property for the longest increasing subsequence problem alluded to in the previous paragraph.  
    \begin{figure*}
\centering
\includegraphics[width=0.7\textwidth]{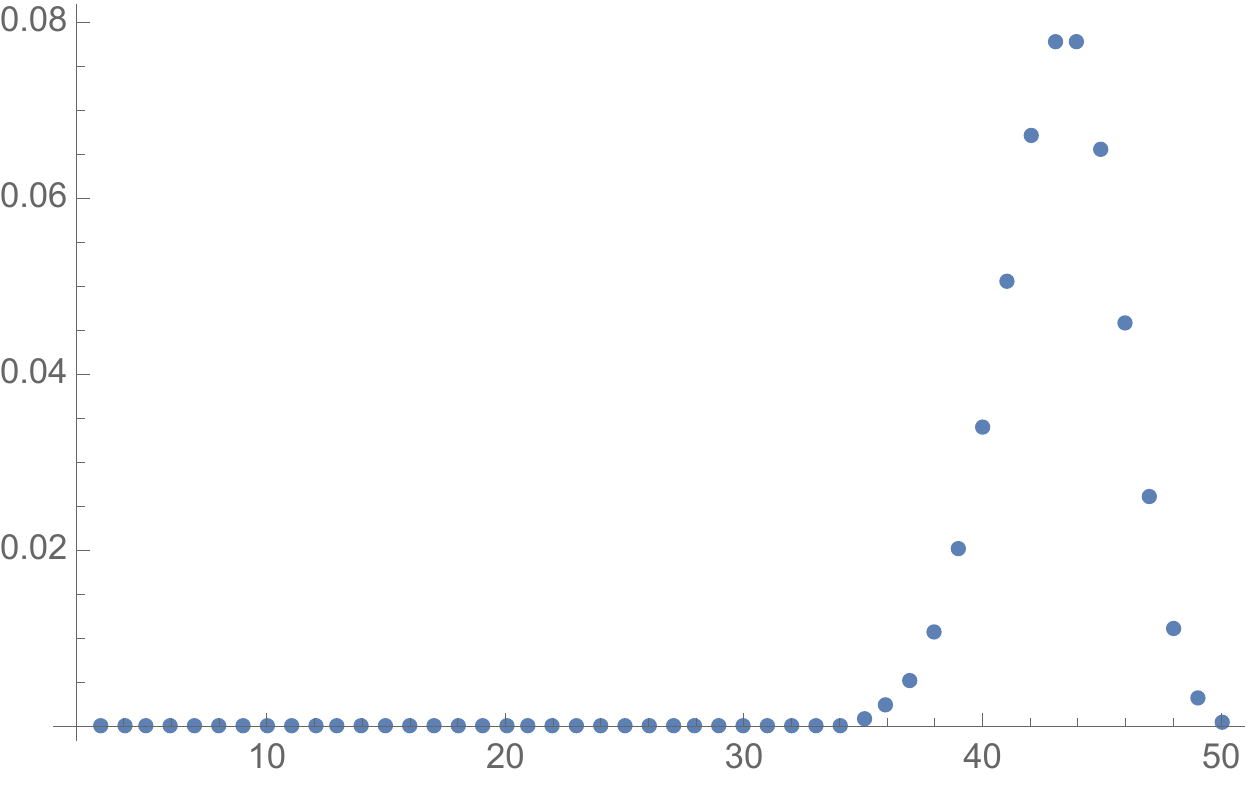}
\caption{Plot of the difference (\ref{E.3}) as a function of $k$ for $N=100$ and $\tau = 0.99$, for which
$\mathbb E_{\mathbb R}^\tau/2 \approx 41.1$ and $ \sigma^2(N_{\mathbb R}^\tau) \approx 25.9$. }
\label{Fig1-23}
\end{figure*}
  
 Below (\ref{SA}) it was noted that the $\tau = 0$ case of the elliptic GinOE corresponds to the GinOE, so the former is a generalisation of the latter.
 Another generalisation of the GinOE is to consider products of $m$ independent GinOE matrices.
  A determinant formula for the generating function of the probabilities of the real eigenvalues, which like (\ref{10.1}) and (\ref{12.1a})
 is derived based on a skew orthogonal polynomial formalism, has been given in \cite[Th~1]{FI16}. While this allows for exact computation of the probabilities 
 for small matrix size in the case $m=2$,  the matrix in the determinant formula is not symmetric and so the location of its zeros are not evident.
  As emphasised in the proof of Proposition 
 \ref{Px}, a key ingredient in deducing a determinant formula involving a symmetric matrix in the skew orthogonal polynomial formalism
 is a derivative formula for $q_{2k+1}(z)$ as in (\ref{q}). In the case of products, while $q_{2k+1}(z)$ remains a simple linear combination of two
 monomials, a derivative formula no longer holds.  Nonetheless, since by the use of hypergeometric polynomial evaluation formulas for the underlyiing
 Meijer G-functions \cite{Ku15} we have available
 the exact probabilities of the real eigenvalues for $m=2$ and small $N$ (specifically $N=6$; see \cite[Table 1]{FI16}), we can compute the location of
 the zeros the generating function numerically. It is found that all the zeros are again on the negative real axis. The mechanism for this, and thus a
 LCLT remains to be found. Using analysis that is independent of the location of the zeros of the generating function, but rather based instead on the Pfaffian point process
 structure of the correlation functions, a CLT in this setting has been established
 in \cite{Si17,FS21}. Further, in \cite{AB23}, formulas from the Pfaffian point process structure have further been used to 
 determine the asymptotic form of $\mathbb E(N_{\mathbb R}^m)$ and $\sigma^2(N_{\mathbb R}^m)$
 (here the use of the superscript $m$ denotes the $m$ products) in the double scaling $N,m \to \infty$ with
 $m/N$ fixed
 limit. 
 
 Exact results for the probability of a prescribed number of real eigenvalues for 
 products of $m$ real random matrices are also known in the case that the individual matrices are constructed from
an $N \times N$  truncation of a Haar distributed $(n+N) \times (n+N)$ real orthogonal matrix \cite{FIK20}. The case $m=1$ was
further considered in the review \cite{BF23}. There, due to a derivative structure for the odd indexed skew orthogonal polynomials
analogous to that in (\ref{q}), a real symmetric matrix formula for the generating function of the probabilities
was shown to hold true, from which the validity of a LCLT follows \cite[paragraph below Prop.~4.7]{BF23}. In the case $n=1$ this generating
function first appeared in the work \cite{PS18}. However, as for the products of GinOE matrices, a derivative structure for the odd indexed skew orthogonal polynomials
 in the cases $m \ge 2$ no longer holds, and the matrix in the determinant formula for the generating function is no longer
symmetric. From the tabulation \cite[Table 1]{FIK20} low order cases (specifically $n=N=4$, $m=2$ and $m=3$) can be explored numerically, with
the result that the zeros are found to be on the negative real axis. It seems that the alternative method of \cite{Si17,FS21} of studying a CLT for the
real eigenvalue fluctuations via the underlying Pfaffian point process structure has yet to be followed through, although the work \cite{LMS22} (see \cite{KSZ09} for the
case $n=1$) has made use of this
structure to compute the asymptotics of the expected number of real eigenvalues in various regimes, in particular $n$ fixed and $n$ proportional to
$N$.

 \subsection*{Acknowledgments}
	The research of PJF is part of the program of study supported
	by the Australian Research Council  Discovery Project grant DP210102887.
	Thanks are due to Leslie Molag for his presentation on \cite{BMS23} in the Bielefeld-Melbourne 
	random matrix theory seminar before its appearance on the arXiv, and to Sung-Soo Byun for
	helpful feedback on a draft of the manuscript.

\providecommand{\bysame}{\leavevmode\hbox to3em{\hrulefill}\thinspace}
\providecommand{\MR}{\relax\ifhmode\unskip\space\fi MR }
\providecommand{\MRhref}[2]{%
  \href{http://www.ams.org/mathscinet-getitem?mr=#1}{#2}
}
\providecommand{\href}[2]{#2}

\end{document}